\documentclass[11pt]{amsart}

\usepackage[T1]{fontenc}
\usepackage[utf8]{inputenc}
\usepackage{lmodern}
\usepackage[a4paper,margin=1.05in]{geometry}
\usepackage{amsmath,amssymb,amsthm,mathtools}
\usepackage[expansion=false]{microtype}
\usepackage{enumitem}
\usepackage[hidelinks]{hyperref}
\hypersetup{
  pdftitle={A Special Pi01 Class with the Join Property but without Pseudojump Inversion},
  pdfauthor={Patrizio Cintioli},
  pdfkeywords={Pi01 classes, Turing degrees, Join Property, pseudojump inversion, generalized lowness}
}

\newtheorem{theorem}{Theorem}[section]
\newtheorem{lemma}[theorem]{Lemma}
\newtheorem{proposition}[theorem]{Proposition}

\theoremstyle{definition}
\newtheorem{definition}[theorem]{Definition}
\newtheorem{remark}[theorem]{Remark}
\numberwithin{equation}{section}

\newcommand{\N}{\mathbb N}
\newcommand{\Cantor}{2^{\mathbb N}}
\newcommand{\Strings}{2^{<\mathbb N}}
\newcommand{\Addr}{\mathcal A}
\newcommand{\leT}{\leq_{\mathrm T}}
\newcommand{\nleT}{\nleq_{\mathrm T}}
\newcommand{\geT}{\geq_{\mathrm T}}
\newcommand{\ltT}{<_{\mathrm T}}
\newcommand{\gtT}{>_{\mathrm T}}
\newcommand{\eqT}{\equiv_{\mathrm T}}
\newcommand{\join}{\mathbin{\oplus}}
\newcommand{\concat}{\mathbin{{}^{\smallfrown}}}
\newcommand{\GL}{\mathrm{GL}_1}
\newcommand{\JP}{\mathsf{JP}}
\newcommand{\PJI}{\mathsf{PJI}}
\newcommand{\dom}{\operatorname{dom}}
\newcommand{\degT}{\operatorname{deg}_{\mathrm T}}
\newcommand{\emptyseq}{\langle\rangle}
\newcommand{\cyl}[1]{[#1]}
\newcommand{\ep}{\theta}

\title[Join property without pseudojump inversion]{A Special $\Pi^0_1$ Class with the Join Property\texorpdfstring{\\}{} but without Pseudojump Inversion}
\author{Patrizio Cintioli}
\address{Mathematics Division, School of Science and Technology, University of Camerino, Italy}
\email{patrizio.cintioli@unicam.it}

\subjclass[2020]{Primary 03D28}
\keywords{$\Pi^0_1$ classes, Turing degrees, Join Property, pseudojump inversion, generalized lowness, finite injury}

\begin{document}

\begin{abstract}
Jananthan and Simpson showed that special $\Pi^0_1$ classes exist in
three of the four possible cases determined by whether the Join
Property and the Pseudojump Inversion Property hold, leaving open the
existence of a special $\Pi^0_1$ class with the Join Property but not
the Pseudojump Inversion Property. We construct such a class, thereby
answering their question affirmatively and completing the four-way
existence pattern.

Moreover, the class can be chosen to consist entirely of generalized
low reals, with the Join Property holding locally on every nonempty
relatively basic open subclass. More generally, every special
$\Pi^0_1$ class contained in $\mathrm{GL}_1$ has a special
$\Pi^0_1$ extension, still contained in $\mathrm{GL}_1$, on which
the Join Property holds locally.

We also show that no class contained in $\mathrm{GL}_1$ has the
Pseudojump Inversion Property. An independent appendix proves that,
for any class with the Join Property, every jump-inversion fibre above
$0'$ is countably infinite.
\end{abstract}
\maketitle

\section{Introduction and statements}

At the level of Turing degrees, the Turing jump induces a map whose
range is exactly the upper cone above $0'$. One inclusion is
immediate, since $0'\leT X'$ for every real $X$, while the converse
is the content of Friedberg's Jump Inversion Theorem: every Turing
degree above $0'$ has a preimage under the jump. A remarkable
refinement, due to Jockusch and Soare \cite[Theorem 1.4]{JS}, says that,
for every real $A\geT0'$ and every prescribed special
$\Pi^0_1$ class $\mathcal C$, some $B\in\mathcal C$ satisfies
$B'\eqT B\join0'\eqT A$.

The Posner--Robinson Join Theorem and the Jockusch--Shore Pseudojump
Inversion Theorem are two closely related classical results concerning
the interaction between the Turing jump, joining, and relative
computable enumerability. This naturally raises the question whether
analogous class-restricted forms of these two theorems hold.

Pseudojump operators were introduced and studied by Jockusch and
Shore as generalizations of the Turing jump. In the form relevant
here, they are the operators
\[
  J_e(X)=X\join W_e^X,
\]
where $W_e^X$ is the $e$th set computably enumerable in $X$.
Their Pseudojump Inversion Theorem shows, in particular, that for
each $e$ there is a noncomputable c.e.\ set $B$ such that
\[
  J_e(B)\eqT0'.
\]
The usefulness of this generalization goes beyond the formal analogy
with the jump: pseudojump inversion provides a flexible method for
constructing Turing degrees with prescribed degree-theoretic
properties.

Jockusch and Shore used this
method to obtain, among other applications, finite-injury
constructions at the various levels of the high and low hierarchies
\cite{JSh}.
Their subsequent work developed finite and transfinite
iterations of these operators through the REA hierarchy, together
with general completeness and join theorems
\cite[Theorems 2.6 and 3.2]{JShII}.

A recurring question in the later theory has been which additional
degree-theoretic constraints can be imposed on an inversion.
Coles, Downey, Jockusch, and LaForte studied which extra properties
can be imposed on sets completing a pseudojump operator and showed,
in particular, that pseudojump completion by c.e. sets is not always compatible
with upper cone avoidance \cite{CDJL}. Further work has continued this
theme of compatibility between pseudojump inversion and additional
requirements; see, for example, \cite{DG}. The problem considered
here is a class-restricted version of the same general phenomenon:
rather than imposing only a degree-theoretic side condition on the
witness, we require the witness to lie in a prescribed special
$\Pi^0_1$ class.

For standard background on classical computability theory and the
c.e.\ degrees, see Soare \cite{Soare}. For jump classes, jump
inversion, and pseudojump operators (called \emph{hops} there), see
Odifreddi \cite[\S\S XI.1, XIII.4]{Odifreddi}.

We now fix notation. We work in Cantor space
$\Cantor=\{0,1\}^{\N}$, where
$\N=\{0,1,2,\ldots\}$, and identify each real $X$ with the subset
$\{n:X(n)=1\}$ of $\N$.

For reals $X,Y$, we write $X\leT Y$, equivalently $Y\geT X$, if
membership in $X$ is decidable by a Turing machine with oracle $Y$.
We use the corresponding strict and equivalence relations:
\[
\begin{aligned}
  X\eqT Y
    &\quad\Longleftrightarrow\quad
      X\leT Y\ \text{and}\ Y\leT X,\\
  X\ltT Y
    &\quad\Longleftrightarrow\quad
      X\leT Y\ \text{and}\ Y\nleT X,\\
  X\gtT Y
    &\quad\Longleftrightarrow\quad
      Y\ltT X.
\end{aligned}
\]
In degree comparisons, $0$ denotes the computable Turing degree,
that is, the degree of $\varnothing$; thus
\[
  X\gtT0
  \quad\Longleftrightarrow\quad
  0\ltT X
  \quad\Longleftrightarrow\quad
  X\text{ is noncomputable}.
\]

The \emph{join} $X\join Y$ is obtained by interleaving the two reals:
\[
  (X\join Y)(2n)=X(n),\qquad
  (X\join Y)(2n+1)=Y(n).
\]
An oracle for $X\join Y$ therefore gives access to both $X$ and $Y$.
Fix a standard effective enumeration $(\Phi_e)_{e\in\N}$ of oracle
Turing machines. The \emph{Turing jump} of $X$ is
\[
  X'=\{e:\Phi_e^X(e)\downarrow\},
\]
where $\downarrow$ denotes convergence. We write $0'=\varnothing'$
for the ordinary halting set.

A set is \emph{computably enumerable} (c.e.) in $X$ if it can be
enumerated by a Turing machine with oracle $X$. Fix a standard
effective enumeration $(W_e^X)_{e\in\N}$ of these sets, and write
$W_e=W_e^{\varnothing}$ in the unrelativised case.

Let $\Strings$ denote the set of finite binary strings. For
$\sigma\in\Strings$ and $X\in\Cantor$, the notation $\sigma\prec X$
means that $\sigma$ is an initial segment of $X$. The corresponding
\emph{cylinder} is
\[
  \cyl{\sigma}=\{X\in\Cantor:\sigma\prec X\}.
\]
These cylinders form a basis for the topology of Cantor space.
A \emph{c.e. open subset of Cantor space} is a union of a computably
enumerable family of cylinders. A $\Pi^0_1$ class is the complement
in Cantor space of such an open set, and is \emph{special} if it is
nonempty and has no computable member.
We write
\[
  \GL=\{X\in\Cantor:X'\eqT X\join0'\}
\]
for the class of \emph{generalized low}, or $\mathrm{GL}_1$, reals.
The subscript $1$ refers to the first Turing jump.

Jananthan and Simpson formulate the following class-restricted
versions of the Join Theorem and the Pseudojump Inversion Theorem.

\begin{definition}[Jananthan--Simpson \cite{JS}]
A class $\mathcal P\subseteq\Cantor$ has the \emph{Join Property},
denoted $\JP(\mathcal P)$, if, whenever
$A\geT Z\join0'$ and $Z\gtT0$, some $B\in\mathcal P$ satisfies
\begin{equation}\label{eq:JP}
  A\eqT B'\eqT B\join0'\eqT B\join Z.
\end{equation}
It has the \emph{Pseudojump Inversion Property}, denoted
$\PJI(\mathcal P)$, if, for every $A\geT0'$ and for all $e\in\N$,
some $B\in\mathcal P$ satisfies
\begin{equation}\label{eq:PJI}
  A\eqT J_e(B)\eqT B\join0'.
\end{equation}
\end{definition}

Thus the Join Property requires the witness in the
Posner--Robinson Join Theorem \cite{PR} to lie in a prescribed
class, while the Pseudojump Inversion Property imposes the
corresponding restriction on the Jockusch--Shore Pseudojump
Inversion Theorem \cite{JSh}.

Jananthan and Simpson realized three of the four logically possible
combinations of these two properties for special $\Pi^0_1$ classes.
The class $\mathrm{CPA}$ of complete consistent extensions of Peano
Arithmetic has both properties. There are special $\Pi^0_1$ classes
with the Pseudojump Inversion Property but without the Join Property,
and they construct a special $\Pi^0_1$ class having neither property.
The remaining possibility was left open:
\[
\begin{array}{c|cc}
 & \PJI & \neg\PJI \\ \hline
\JP
  & \mathrm{CPA}
  & \text{open in \cite{JS}}\\
\neg\JP
  & \text{known}
  & \text{known}.
\end{array}
\]
Equivalently, they ask whether there exists a special $\Pi^0_1$
subclass of Cantor space which has the Join Property but not the
Pseudojump Inversion Property \cite[Introduction]{JS}.

Our main theorem supplies precisely this missing case and therefore
completes the four-way existence pattern for the two properties.

\begin{theorem}\label{thm:main}
There is a special $\Pi^0_1$ class $\mathcal P\subseteq\Cantor$ such that
\[
  \mathcal P\subseteq\GL,\qquad
  \JP(\mathcal P),\qquad
  \neg\PJI(\mathcal P).
\]
More strongly, the Join Property holds locally on $\mathcal P$:
for every finite binary string $\sigma$ such that
$\mathcal P\cap\cyl{\sigma}\ne\varnothing$, the class
$\mathcal P\cap\cyl{\sigma}$ has the Join Property.
\end{theorem}

The generalized-lowness requirement plays two roles in the argument.
First, the construction ensures that every member of the limit
class $\mathcal P$ is generalized low. Second, generalized lowness
provides the separation: no class contained in $\GL$ can have
the Pseudojump Inversion Property. Thus the main work is to
construct a special $\Pi^0_1$ class with the Join Property
whose members all belong to $\GL$.

In fact, the construction gives a stronger extension theorem.

\begin{theorem}[Extension theorem]\label{thm:extension}
Let $\mathcal Q\subseteq\Cantor$ be a special $\Pi^0_1$ class with
$\mathcal Q\subseteq\GL$. There is a special $\Pi^0_1$ class
$\mathcal P$ such that
\[
  \mathcal Q\subseteq\mathcal P\subseteq\GL,
\]
and, for every $\sigma\in\Strings$ such that
$\mathcal P\cap\cyl{\sigma}\ne\varnothing$, the class
$\mathcal P\cap\cyl{\sigma}$ has the Join Property.
\end{theorem}

A class $\mathcal Q$ satisfying the hypothesis exists by
\cite[Theorem 6.2, parts 2 and 4]{JS}. Only generalized lowness of
the individual members is needed here, not the stronger assertion
there concerning finite joins.

The construction uses a computable decreasing sequence of clopen
classes. Copies of the initial class $\mathcal Q$ are retained at
the limiting points of an infinitely branching system of choices.
 A finite-injury mechanism eventually decides, at each address and
exit, whether the relevant c.e.\ open requirement is met;  a permanent
historical record of the actions, indexed by epochs, later permits the
required decoding. The two decoding procedures are arranged so that the
same constructed path $B$ satisfies
\[
  A\eqT B\join Z\eqT B\join0'\eqT B'.
\]
The construction is local in the sense expressed in
Theorem~\ref{thm:main}: every nonempty relatively basic open part of
$\mathcal P$ again has the Join Property.

We prove the extension theorem in
Sections~\ref{sec:exits}--\ref{sec:coding}.
Section~\ref{sec:GL} establishes generalized lowness and specialness,
and Section~\ref{sec:consequences} derives the failure of pseudojump
inversion and completes the proof of the main theorem.
The appendix is independent of the construction: it shows that, for
any class with the Join Property, every jump-inversion fibre above
$0'$ represents countably infinitely many Turing degrees.

\section{Exits, addresses, and requirements}\label{sec:exits}

For a finite sequence $\sigma$, its length is denoted by $|\sigma|$,
and $\emptyseq$ denotes the empty sequence. All coordinates are
indexed starting at $0$. For finite sequences over the same alphabet,
$\sigma\preceq\tau$ means that $\sigma$ is a prefix of $\tau$, while
$\sigma\prec\tau$ means that it is a proper prefix. We use $\succeq$
for the reverse prefix relation. Two finite sequences are
\emph{incompatible} if neither is a prefix of the other.
For a finite or infinite sequence $\alpha$, the notation
$\alpha\upharpoonright n$ denotes its prefix of length $n$, whenever
defined.

The operation $\concat$ denotes concatenation. For binary strings,
$\sigma\concat b$ appends the single bit $b$, and $\sigma\concat X$
prefixes the real $X$ by $\sigma$. More generally, for a class
$\mathcal C\subseteq\Cantor$, put
\[
  \sigma\concat \mathcal C=\{\sigma\concat X:X\in \mathcal C\}.
\]
Thus $\sigma\concat \mathcal C$ is a prefixed copy of $\mathcal C$ inside
$\cyl{\sigma}$. We write $0^\omega$ for the constant zero real.

A set is \emph{clopen} if it is both open and closed; in Cantor
space, these are exactly the finite unions of cylinders.
A binary tree is a subset of $\Strings$ closed under prefixes.
For such a tree $T$, let
\[
  [T]=\{X\in\Cantor:(\forall n)\ X\upharpoonright n\in T\}
\]
be its class of infinite paths.

Fix $\mathcal Q$ as in Theorem~\ref{thm:extension}. Choose a computable
tree $T\subseteq\Strings$ with $\mathcal Q=[T]$ and $\emptyseq\in T$.
We do not assume that $T$ is pruned: some nodes of $T$ may lie on no
infinite path.

For a nonempty string $\eta$,
write $\eta^{-}=\eta\upharpoonright(|\eta|-1)$ for its immediate
predecessor. Let
\begin{equation}\label{eq:exits}
  E=\{\eta\in\Strings\setminus T:\eta^{-}\in T\}.
\end{equation}

The members of $E$ are called \emph{exits} from $T$.
A set of strings is \emph{prefix-free} if no two distinct members
are comparable under $\preceq$.

\begin{lemma}\label{lem:exits}
The set $E$ is computable, infinite, and prefix-free. Every real either
belongs to $\mathcal Q$ or extends exactly one member of $E$. If
$\mathcal Q\cap\cyl{\delta}\ne\varnothing$, some member of $E$ extends $\delta$.
\end{lemma}
\begin{proof}
Computability follows from that of $T$. A string in $E$ is the first
prefix outside $T$ of any real extending it, which proves prefix-freeness
and the asserted partition. If $E$ were finite, then
$\mathcal Q=\Cantor\setminus\bigcup_{\eta\in E}\cyl{\eta}$ would be a nonempty
clopen class and would contain a computable real, a contradiction.
Finally, if $\mathcal Q\cap\cyl{\delta}\ne\varnothing$, then $\delta\in T$.
The computable real $\delta\concat0^\omega$ does not belong to $\mathcal Q$.
Its first prefix outside $T$ therefore strictly extends $\delta$ and
belongs to $E$.
\end{proof}

Fix a computable enumeration without repetitions
$E=\{\eta_i:i\in\N\}$. Given an oracle for a real outside $\mathcal Q$, its
exit index $i$ can be recovered: find its shortest prefix outside $T$
and search the enumeration for that prefix.

Our \emph{addresses} are elements of
\[
  \Addr=(\N\times\{0,1\})^{<\N}.
\]
The length $|a|$ of an address is the number of pairs occurring in it.

The same prefix conventions
apply to addresses. We use fixed computable codings of binary strings,
addresses, and finite tuples by natural numbers whenever these objects
are used as inputs to algorithms.
We write $a\concat(k,b)$ for the address obtained by appending the
pair $(k,b)$ to $a$, where $k\in\N$ and $b\in\{0,1\}$.
Here $k$ records an exit index and $b$ is a binary choice.

A \emph{slot} is a pair $(a,k)$, where $a\in\Addr$ and $k\in\N$.
The two immediate child addresses associated with the slot $(a,k)$ are
\[
  a\concat(k,0)
  \qquad\text{and}\qquad
  a\concat(k,1).
\]
For example, if
\[
  a=\langle(4,0),(7,1)\rangle,
\]
then the slot $(a,3)$ has the two immediate child addresses
\[
  a\concat(3,0)
  \qquad\text{and}\qquad
  a\concat(3,1).
\]
Thus the natural number $3$ records the chosen exit, while the last
coordinate $0$ or $1$ records the subsequent binary choice.

A slot $(c,\ell)$ is \emph{below} $(a,k)$ if its address $c$
extends one of these two child addresses; equivalently,
\begin{equation}\label{eq:below}
  a\concat(k,b)\preceq c
  \quad\text{for some }b\in\{0,1\}.
\end{equation}
This relation is computable. In particular, the slot $(a,k)$ itself
is not below $(a,k)$.

A slot $(d,j)$ is a \emph{proper ancestor slot} of an address $a$ if
$a$ extends one of the two child addresses associated with $(d,j)$;
equivalently,
\[
  d\concat(j,b)\preceq a
  \quad\text{for some }b\in\{0,1\}.
\]
Each address $a$ will be represented in the ambient binary tree
$\Strings$ by a finite binary string, called its \emph{physical root}.
During the construction its value may change; its value at stage $s$
will be denoted by $p_s(a)$. We shall prove that $p_s(a)$ eventually
stabilises and denote its final value by $p(a)$. The binary-string
length $|p_s(a)|$ (and, eventually, $|p(a)|$) should not be confused
with the address length $|a|$.
Its role is to locate, inside Cantor space, the prefixed copy of $\mathcal Q$
associated with the address $a$.

We use the following uniformly c.e. sequence $(U_n)_{n\in\N}$ of
open sets, meaning that one algorithm, given $n$, enumerates cylinders
whose union is $U_n$:
\begin{align}
 U_{2e}&=\{X:\Phi_e^X(e)\downarrow\},\label{eq:jumpopens}\\
 U_{2e+1}&=\{X:\exists m\,[\varphi_e(m)\downarrow\in\{0,1\}
                         \ \&\ X(m)\ne\varphi_e(m)]\}.
                         \label{eq:diagonalopens}
\end{align}
Here $(\varphi_e)$ is a standard effective enumeration of partial
computable functions.
The notation
$\varphi_e(m)\downarrow\in\{0,1\}$ means that the computation
converges with a binary output.

Fix uniformly computable increasing
clopen approximations $U_{n,s}$ with $U_n=\bigcup_s U_{n,s}$.
Concretely, a finite list of cylinders describing $U_{n,s}$ can be
computed from $n,s$, and $U_{n,s}\subseteq U_{n,s+1}$.
Slots at addresses of length $n$ will decide
$U_n$ on their surviving paths. The even requirements control the jump;
the odd requirements exclude computable complete paths.

\section{The computable construction}\label{sec:construction}

The construction proceeds in stages. At each stage we keep a finite
table recording the actions which are still current, together with a
permanent record of all actions that have ever occurred. The current
table determines the locations of the copies of $\mathcal Q$ inside Cantor
space and hence the current clopen class $\mathcal P_s$. At each stage one
slot is visited.
If it acts, the construction restricts the corresponding exit region
to a chosen cylinder and reinitialises all descendant slots.

\subsection{The data at stage \texorpdfstring{$s$}{s}}

At the beginning of stage $s+1$ we have a finite partial table $D_s$.
Its domain $\dom(D_s)$ is the finite set of slots with a current
stored entry. An entry
\[
  D_s(a,i)=(\rho,u)
\]
records the stem $\rho$ selected by an action at stage $u>0$ which is
still current. Every action has its own stage number as identifier;
identifiers are never reused. An action at an ancestor deletes the
current entries in the slots below it.

There is also a permanent historical record of events
\begin{equation}\label{eq:event}
  (a,t,i,\rho,u).
\end{equation}
This says that $(a,i)$ acted at stage $u$, during epoch $t$ of $a$, and
selected $\rho$. Nothing is ever removed from this record. Its purpose
is distinct from that of $D_s$: an obsolete action disappears from the
current table but remains in the historical record.
From the current table $D_s$ we derive four kinds of data. The string
$p_s(a)$ is the current physical root of the address $a$;
$r_s(a,i)$ is the current exit root obtained by appending $\eta_i$
to $p_s(a)$; $g_s(a,i)$ is the current end of the stem attached to
that exit; and $\ep_s(a)$ records the current version, or epoch, of
the address.

Formally, these data are defined from $D_s$ by recursion on address
length. Set
\[
  p_s(\emptyseq)=\emptyseq,\qquad \ep_s(\emptyseq)=0,
\]
and put
\begin{align}
 r_s(a,i)&=p_s(a)\concat\eta_i,\label{eq:stage-root}\\
 g_s(a,i)&=
 \begin{cases}
   \rho,&D_s(a,i)=(\rho,u),\\
   r_s(a,i),&(a,i)\notin\dom(D_s),
 \end{cases}\label{eq:stage-stem}\\
 p_s(a\concat(i,b))&=g_s(a,i)\concat b,\label{eq:stage-child}\\
 \ep_s(a\concat(i,b))&=
 \begin{cases}
   u,&D_s(a,i)=(\rho,u),\\
   \ep_s(a),&(a,i)\notin\dom(D_s).
 \end{cases}\label{eq:stage-epoch}
\end{align}
Only the finitely many prefixes of $a$ are consulted to evaluate these
functions on a given input. We maintain that every current stem
extends its current exit root.
The values $p_s(a)$, $r_s(a,i)$, and $g_s(a,i)$ are finite binary
strings, whereas $\ep_s(a)\in\N$ identifies the current epoch.
This label need not equal the current stage $s$ and does not measure
the length of any of these strings.

The current clopen class is given by the finite table:
\begin{equation}\label{eq:table-class}
\mathcal P_s=\Cantor\setminus
       \bigcup_{(a,i)\in\dom(D_s)}
       \bigl(\cyl{r_s(a,i)}\setminus\cyl{g_s(a,i)}\bigr).
\end{equation}

Under the maintained stem condition, a current record at $(a,i)$
excludes from the exit cylinder $\cyl{r_s(a,i)}$ everything outside
the stem cylinder $\cyl{g_s(a,i)}$; if the two strings are equal,
this exclusion is empty.

A clopen set here is given by an explicit finite Boolean combination of
cylinders, so emptiness and inclusion are decidable.

\subsection{Initialisation and scheduling}

Initially $D_0$ and the historical record are empty, and $\mathcal P_0=\Cantor$.
We use a simple fair schedule: exactly one slot is visited at each
stage, and every slot is visited infinitely often.
Fix a computable bijection $\nu$ from $\N$ to the slots. At stage $s+1$
visit the slot
\[
  \nu(v_2(s+1)),
\]
where $v_2(m)$ is the exponent of $2$ in the positive integer $m$.
There is therefore at most one action per stage. Each individual
stage terminates: all tests performed at that stage are decidable,
and any search undertaken after a positive test is guaranteed to
succeed. We impose no a priori running-time bound.

\subsection{The action}

An action has two effects: it selects a cylinder contained in the
current exit region, and it reinitialises the descendant slots by
deleting their current entries. The corresponding historical events
are retained in the permanent record.

Suppose the visited slot is $(a,i)$. If it is already in $\dom(D_s)$,
do nothing. Otherwise put
\[
  n=|a|,\qquad t=\ep_s(a),\qquad r=r_s(a,i),
\]
and test whether the clopen set
\begin{equation}\label{eq:positive-test}
         C=\mathcal P_s\cap\cyl r\cap U_{n,s}
\end{equation}
is nonempty. If it is empty, again do nothing. If it is nonempty,
choose, in length-lexicographic order, a string $\rho\succeq r$ with
\begin{equation}\label{eq:available-cylinder}
  \cyl\rho\subseteq C.
\end{equation}
Such a string exists because $C$ is nonempty and clopen. The search is
computable, since inclusion of clopen sets is decidable.

Perform the following update. Delete from $D_s$ all entries in slots
below $(a,i)$, retain all other entries, and insert
\[
  D_{s+1}(a,i)=(\rho,s+1).
\]
Add $(a,t,i,\rho,s+1)$ to the permanent historical record. The intended
update to the clopen class is
\begin{equation}\label{eq:clopen-update}
 \mathcal P_{s+1}=(\mathcal P_s\setminus\cyl r)\cup\cyl\rho.
\end{equation}
Lemma~\ref{lem:invariants} below verifies that the new table describes
exactly this class. If no action is performed, both the table and the
clopen class remain unchanged.

Every action counts as an action even if $\rho=r$. In particular, it
still receives a new identifier and reinitialises the descendants.
The epoch is a version identifier, not a code for the physical string.

\section{Invariants, finite injury, and decisions}\label{sec:verification}

The verification has four steps. First we check the stage invariants
and the monotonicity of the clopen approximations. Next we prove finite
injury and obtain final roots, stems, and epochs. We then describe the
paths through the limit class. Finally, we show that each exit at a
final epoch decides the corresponding c.e. open requirement.

\subsection{Stage geometry and invariants}

\begin{lemma}[Stage invariants]\label{lem:invariants}
The construction is computable, and at every stage the following
properties hold:
\begin{enumerate}[label=\textup{(\roman*)},leftmargin=*]
\item every stored stem extends its current exit root;
\item the new table at an action gives precisely
      \eqref{eq:clopen-update}, and hence $\mathcal P_{s+1}\subseteq\mathcal P_s$;
\item for every $a\in\Addr$,
\begin{equation}\label{eq:stage-copies}
  p_s(a)\concat\mathcal Q\subseteq\mathcal P_s;
\end{equation}
\item for every $a\in\Addr$ and $i\in\N$, the exit $i$ at $a$
      has a surviving continuation:
\begin{equation}\label{eq:stage-nonempty-exit}
  \varnothing\ne p_s(a\concat(i,0))\concat \mathcal Q
       \subseteq\mathcal P_s\cap\cyl{r_s(a,i)}.
\end{equation}
\end{enumerate}
\end{lemma}

\begin{proof}
We argue by induction on the stages. Initially $D_0=\varnothing$
and $\mathcal P_0=\Cantor$, so the stem condition and the two assertions
about copies of $\mathcal Q$ hold.

\smallskip
\noindent\emph{Geometric preliminaries.}
First observe some geometric consequences of the stem condition.
For every address $a$, exit $i$, and bit $b<2$,
\[
  p_s(a)\concat\eta_i
  =r_s(a,i)
  \preceq g_s(a,i)
  \prec p_s(a\concat(i,b)).
\]
Consequently, if $a\concat(i,b)\preceq c$, then
$g_s(a,i)\concat b\preceq p_s(c)$. The two children of a slot
have incompatible roots. Distinct exits at the same address
also have incompatible roots, because $E$ is prefix-free.
It follows, by considering the first pair at which two
addresses differ, that incompatible addresses have disjoint
root cylinders.

\smallskip
\noindent\emph{Effect of an action.}
Suppose an action is performed at $(a,i)$ at stage $s+1$.
By the action rule, $(a,i)\notin\dom(D_s)$. Put
$r=r_s(a,i)$, and let $\rho\succeq r$ be the chosen string, so
\[
  \cyl{\rho}\subseteq
 \mathcal P_s\cap\cyl r\cap U_{|a|,s}.
\]
Partition the old records into those retained and those deleted.
The point is to show that deleting the descendant records does not
restore any real which had already been removed from $\mathcal P_s$.
All exclusions considered in this paragraph are evaluated
using the old physical roots and stems.

A retained record is either at an ancestor slot traversed on the way
to $a$, or lies outside the descendant region of $(a,i)$. In the
latter case its root is incompatible with the exit root $r$.
Its root depends only on records along its own
address, none of which is changed by this action. Thus its
physical root and its stored stem remain unchanged. Its
excluded set is disjoint from $\cyl r$: in the ancestor case,
$\cyl r$ is contained in its retained stem cylinder; in the
incompatible case, the two root cylinders are disjoint.

Every deleted record is below $(a,i)$, so its old root
cylinder, and hence its excluded set, is contained in
$\cyl r$. Its excluded set is also disjoint from
$\cyl{\rho}$, since $\cyl{\rho}\subseteq\mathcal P_s$.
Therefore every deleted exclusion is contained in
\[
  H=\cyl r\setminus\cyl{\rho}.
\]

The new table satisfies the stem condition. Indeed, all
retained records have unchanged roots and stems. The root of
$(a,i)$ itself is still $r$, since its root depends only on
proper ancestors, and its new stem is $\rho\succeq r$.
All old records below $(a,i)$ have been deleted.

Let $F$ be the union of the exclusions of the retained
records, and let $G$ be the union of the exclusions of the
deleted records. Then
\[
  \Cantor\setminus\mathcal P_s=F\cup G
  \qquad\text{and}\qquad
  G\subseteq H.
\]
Thus every exclusion which disappears from the table is already
absorbed by the new exclusion $H$.
The new table has exactly the retained exclusions together
with the new exclusion $H$. Hence it describes
\[
\begin{aligned}
  \Cantor\setminus(F\cup H)
  &=\Cantor\setminus(F\cup G\cup H)\\
  &=\mathcal P_s\setminus H\\
  &=(\mathcal P_s\setminus\cyl r)\cup\cyl{\rho},
\end{aligned}
\]
where the last equality uses $\cyl{\rho}\subseteq\mathcal P_s$.
This proves \eqref{eq:clopen-update} and monotonicity.
If no action is performed, the table and the class are
unchanged.

\smallskip
\noindent\emph{Survival of the copies of $\mathcal Q$.}
We next verify the copy assertions for any stage $t$ whose
table satisfies the stem condition.
Fix an address $c$ and
a stored slot $(d,j)$.

If $d\concat(j,b)\preceq c$ for some $b<2$, then
\[
  \cyl{p_t(c)}\subseteq\cyl{g_t(d,j)},
\]
so the exclusion at $(d,j)$ does not meet
$p_t(c)\concat \mathcal Q$.

If $d=c$, its exclusion is contained in
$\cyl{p_t(c)\concat\eta_j}$, which is disjoint from
$p_t(c)\concat \mathcal Q$ because $\eta_j\notin T$ and $\mathcal Q=[T]$.
If $c\prec d$, let $(k,b)$ be the first pair of $d$ after
$c$. The exclusion at $(d,j)$ is then contained in
$\cyl{p_t(c)\concat\eta_k}$, and is again disjoint from
$p_t(c)\concat \mathcal Q$.

In all remaining cases, the slot $(d,j)$ lies in a region
incompatible with the address $c$, so its root cylinder is
disjoint from $\cyl{p_t(c)}$.
Thus no exclusion in the table meets $p_t(c)\concat \mathcal Q$.
Formula \eqref{eq:table-class} proves
\eqref{eq:stage-copies}.

Taking $c=a\concat(i,0)$ gives
\eqref{eq:stage-nonempty-exit}, since
\[
  p_t(c)=g_t(a,i)\concat0\succeq r_t(a,i)
\]
and $\mathcal Q$ is nonempty.

\smallskip
\noindent\emph{Effectivity.}
Finally, all stage data are computable. The table is finite,
roots and epochs are evaluated by recursion on finite
addresses, and the relation specifying which records to
delete is computable. The set tested at an action is an
explicitly presented clopen set, so its emptiness and the
inclusion of any proposed cylinder are decidable. If the
test is positive, this clopen set contains a cylinder, so
the length-lexicographic search for $\rho$ terminates.
The finite table update and the addition to the historical
record are effective. This completes the induction.
\end{proof}

\subsection{Finite injury and final values}
The epochs serve as finite-injury version numbers: an address receives
a new epoch exactly when a proper ancestor slot acts, thereby
reinitialising it.

\begin{lemma}[Epochs and stabilisation]\label{lem:stabilisation}
For each address $a$, the epoch $\ep_s(a)$ changes only when a proper
ancestor slot acts, and then changes to that action's new identifier.
It never returns to an earlier value. Each slot acts at most once in
each epoch of its address. Every address has only finitely many epochs.
In particular, the limits
\begin{equation}\label{eq:limits}
 p(a)=\lim_s p_s(a),\qquad
 \ep(a)=\lim_s\ep_s(a),\qquad
 g(a,i)=\lim_s g_s(a,i)
\end{equation}
exist for all $a,i$.
\end{lemma}

\begin{proof}
\noindent\emph{Changes of epoch.}
When $(a,i)$ acts at stage $u$, every table entry below it is deleted.
The epoch recursion therefore gives epoch $u$ to all its descendant
addresses. This identifier exceeds all earlier identifiers. Epochs
outside this descendant region are unchanged. In particular an action
at $(a,i)$ does not change the epoch of $a$ itself. Physical roots can
change only in the same descendant region. Thus a root is constant
throughout a given epoch of its address, even though an action need not
change its physical string.

\smallskip
\noindent\emph{Finite injury.}
Once a slot has acted, its entry remains until a proper ancestor acts.
Such an action changes the epoch of its address. Hence the slot acts
at most once per epoch.

The root address has a single epoch, namely $0$, so each of its
slots acts at most once.
Inductively, if $a$ has finitely many epochs,
then a fixed slot $(a,i)$ acts finitely often. Its actions, together
with the finitely many changes of the epoch of $a$, account for all
reinitialisations of each child $a\concat(i,b)$. The child therefore
has finitely many epochs as well.
This completes the induction on address length.

\smallskip
\noindent\emph{Final values.}
For each fixed address $a$, its epoch is eventually constant.
Since its physical root is constant throughout each epoch,
$p_s(a)$ is eventually constant as well. Fix also an exit $i$.
During the final epoch of $a$, the exit root
$r_s(a,i)=p_s(a)\concat\eta_i$ is fixed, and the slot $(a,i)$
acts at most once. If it acts, its new record can never
subsequently be deleted, since such a deletion would require
a further change of the epoch of $a$. Thus $g_s(a,i)$ is
eventually constant. This proves all the limits in
\eqref{eq:limits}.

The stabilisation is pointwise: no simultaneous stabilisation
of all addresses or slots at a level is asserted.

\end{proof}

The notation without a stage subscript, namely $p(a)$, $\ep(a)$,
and $g(a,i)$, will henceforth denote these final values.
The limits in \eqref{eq:limits} mean eventual constancy at each fixed
input; they are not unions of the stage approximations.

\subsection{The limit class and its paths}

We now pass to the limit and define
\begin{equation}\label{eq:limit-class}
 \mathcal P=\bigcap_s\mathcal P_s.
\end{equation}

It is a $\Pi^0_1$ class, since the complements of the uniformly
computable clopen sets $\mathcal P_s$ enumerate an effectively open set.
We first verify that every final copy of $\mathcal Q$ survives, and then use
these copies to describe all paths through $\mathcal P$.
\begin{lemma}[Surviving copies and paths]\label{lem:paths}
For every address $a$,
\begin{equation}\label{eq:limit-copies}
  p(a)\concat\mathcal Q\subseteq \mathcal P.
\end{equation}
In particular $\mathcal Q\subseteq\mathcal P$ and $\mathcal P$ is nonempty. Every $B\in\mathcal P$ has
exactly one of the following two descriptions under the
parsing rule below:
\begin{enumerate}[label=\textup{(\roman*)},leftmargin=*]
\item the procedure stops at an address $a$, and
      $B=p(a)\concat R$ for some $R\in \mathcal Q$;
\item the procedure continues at every level, giving addresses
      $a_0=\emptyseq$ and $a_{n+1}=a_n\concat(i_n,b_n)$ with
      $p(a_n)\prec B$ for all $n$.
\end{enumerate}
We call these paths \emph{stopped} and \emph{complete}, respectively.
This is a structural description; decidability of membership in $\mathcal Q$ is
not assumed. Conversely, any infinite chain of addresses of the form in
\textup{(ii)} gives a member of $\mathcal P$ by taking the union of its final
roots.
\end{lemma}

\begin{proof}
\noindent\emph{Survival of the final copies.}
Fix an address $a$. By Lemma~\ref{lem:stabilisation}, there is a
stage $s_a$ such that $p_s(a)=p(a)$ for every $s\geq s_a$.
Then \eqref{eq:stage-copies} gives
\[
  p(a)\concat\mathcal Q\subseteq\mathcal P_s
  \qquad(s\geq s_a).
\]
For $s<s_a$, monotonicity gives $\mathcal P_{s_a}\subseteq\mathcal P_s$, so the
same inclusion holds. Hence
\[
  p(a)\concat\mathcal Q\subseteq\bigcap_s\mathcal P_s=\mathcal P,
\]
proving \eqref{eq:limit-copies}. Since
$p(\emptyseq)=\emptyseq$ and $\mathcal Q$ is nonempty, we also have
$\mathcal Q\subseteq\mathcal P$ and $\mathcal P\ne\varnothing$.

\smallskip
\noindent\emph{Parsing a member of $\mathcal P$.}
By the stage recursions, the stem condition, and pointwise
stabilisation, for every $a,i$, and $b<2$,
\[
  p(a)\concat\eta_i\preceq g(a,i),
  \qquad
  p(a\concat(i,b))=g(a,i)\concat b.
\]
We also claim that
\begin{equation}\label{eq:paths-exit-containment}
  \mathcal P\cap\cyl{p(a)\concat\eta_i}\subseteq\cyl{g(a,i)}
  \qquad\text{for every }a,i.
\end{equation}
To see this, choose $s$ sufficiently large that
$p_s(a)=p(a)$ and $g_s(a,i)=g(a,i)$. If
$(a,i)\notin\dom(D_s)$, then
$g(a,i)=r_s(a,i)=p(a)\concat\eta_i$, and the claim is
immediate. Otherwise the table excludes
\[
  \cyl{r_s(a,i)}\setminus\cyl{g_s(a,i)}
\]
from $\mathcal P_s$. Since $\mathcal P\subseteq\mathcal P_s$,
\eqref{eq:paths-exit-containment} follows in this case as well.

Now let $B\in\mathcal P$ and start the parsing at $\emptyseq$.
Suppose it has reached $a$ with $p(a)\prec B$, and write
$B=p(a)\concat R_a$. If $R_a\in \mathcal Q$, the parsing stops.
Otherwise Lemma~\ref{lem:exits} gives a unique $i$ such that
$\eta_i\prec R_a$. By \eqref{eq:paths-exit-containment},
$B$ extends $g(a,i)$. The bit
\[
  b=B(|g(a,i)|)
\]
is therefore defined, and
\[
  p(a\concat(i,b))=g(a,i)\concat b\prec B.
\]
Thus the next address is uniquely determined. This establishes
exhaustiveness and uniqueness of the parsing description.
The parsing is understood structurally; no decision procedure
for membership in $\mathcal Q$ is being asserted.

\smallskip
\noindent\emph{From an infinite address chain to a path.}
Conversely, let $a_0=\emptyseq$ and
$a_{n+1}=a_n\concat(i_n,b_n)$ be any infinite address chain.
The final-root identities give
\[
  p(a_n)\concat\eta_{i_n}
  \preceq g(a_n,i_n)
  \prec p(a_{n+1}).
\]
Since every exit $\eta_i$ is nonempty,
\[
  |p(a_{n+1})|
  \geq |p(a_n)|+|\eta_{i_n}|+1
  \geq |p(a_n)|+2.
\]
The roots therefore have an infinite union
\[
  B=\bigcup_n p(a_n).
\]
Fix $R\in \mathcal Q$. By \eqref{eq:limit-copies}, every
$p(a_n)\concat R$ belongs to $\mathcal P$. These reals converge to $B$,
because they agree with $B$ on the increasingly long prefixes
$p(a_n)$. The class $\mathcal P$ is closed, so $B\in\mathcal P$.

For every $n$, the tail of $B$ after $p(a_n)$ extends
$\eta_{i_n}$ and hence does not belong to $\mathcal Q$. Moreover, the
bit of $B$ immediately after $g(a_n,i_n)$ is $b_n$.
Its unique parsing therefore follows the specified chain
and never stops.
\end{proof}

\subsection{Event sets and final decisions}

For a fixed address and epoch, we now record which exits actually
acted during that epoch.

\begin{definition}[The event sets]\label{def:eventsets}
For $a\in\Addr$ and $t\in\N$, let
\begin{equation}\label{eq:eventsets}
 S_{a,t}=\{i:\exists\rho,u\ (a,t,i,\rho,u)
                         \text{ occurs in the historical record}\}.
\end{equation}
\end{definition}

The family $(S_{a,t})_{a,t}$ is uniformly c.e., including when $t$ is
not an epoch actually attained by $a$. By the $s$-$m$-$n$ theorem,
fix a computable function $q$ such that $S_{a,t}=W_{q(a,t)}$
for all $a,t$. By Lemma~\ref{lem:stabilisation}, for a fixed
triple $(a,t,i)$ there is at most one recorded action.
Once $i$ is enumerated into $S_{a,t}$, the corresponding stem and
identifier can be recovered effectively from the historical record.
Thus $S_{a,t}$ records the positive decisions made at address $a$
during epoch $t$.

At a final epoch, no further injury by a proper ancestor slot is possible.
Consequently, each exit eventually makes a permanent positive or negative
decision about the corresponding open set $U_{|a|}$.

\begin{lemma}[Decision at a final epoch]\label{lem:decision}
Fix an address $a$ and an exit $i\in\N$. Put $t=\ep(a)$ and
$r=p(a)\concat\eta_i$.
Then
\begin{align}
 i\in S_{a,t}&\ \Longrightarrow\ \mathcal P\cap\cyl r\subseteq U_{|a|},
                      \label{eq:positive-decision}\\
 i\notin S_{a,t}&\ \Longrightarrow\ \mathcal P\cap\cyl r\cap U_{|a|}
                      =\varnothing.\label{eq:negative-decision}
\end{align}
If $i\in S_{a,t}$ and its event has stem $\rho$ and identifier $u$,
then $g(a,i)=\rho$ and $\ep(a\concat(i,b))=u$ for $b<2$.
If $i\notin S_{a,t}$, then $g(a,i)=r$ and
$\ep(a\concat(i,b))=t$ for $b<2$.
\end{lemma}

\begin{proof}
\noindent\emph{Beginning of the final epoch.}
Let $s_0$ be the first stage at which $a$ has its final epoch
$t=\ep(a)$. By Lemma~\ref{lem:stabilisation}, for every
$s\geq s_0$,
\[
  \ep_s(a)=t,\qquad p_s(a)=p(a),\qquad r_s(a,i)=r.
\]
No proper ancestor slot of $a$ acts after stage $s_0$.

At stage $s_0$, there is no current entry at any slot $(a,j)$.
Indeed, if $s_0=0$, the initial table is empty. If $s_0>0$,
the action introducing this epoch is at a proper ancestor
slot of $a$ and deletes all entries at slots $(a,j)$.
Consequently, every entry subsequently present at $(a,i)$
comes from an action recorded in $S_{a,t}$, and such an
entry is never subsequently cancelled.

\smallskip
\noindent\emph{Positive decision.}
Suppose first that $i\in S_{a,t}$, and let
$(a,t,i,\rho,u)$ be its unique event. At stage $u$, the
exit root is $r$, and the action chooses
\[
  \cyl{\rho}\subseteq
 \mathcal P_{u-1}\cap\cyl r\cap U_{|a|,u-1}.
\]
By \eqref{eq:clopen-update} and monotonicity,
\[
  \mathcal P\cap\cyl r\subseteq\cyl{\rho}\subseteq U_{|a|}.
\]
This proves \eqref{eq:positive-decision}.
The record $D_s(a,i)=(\rho,u)$ remains present for all
$s\geq u$. The stem and epoch recursions therefore give
\[
  g(a,i)=\rho,\qquad
  \ep(a\concat(i,b))=u
  \quad(b<2).
\]

\smallskip
\noindent\emph{Negative decision.}
The negative case requires an argument: the absence of an action
during the final epoch must be shown to imply that
$\mathcal P\cap\cyl r$ is disjoint from $U_{|a|}$.

Suppose that $i\notin S_{a,t}$. There is no current
entry at $(a,i)$ at any stage $s\geq s_0$.
Hence the
default clauses of the recursions give
\[
  g_s(a,i)=r,\qquad
  \ep_s(a\concat(i,b))=t
  \quad(s\geq s_0,\ b<2).
\]
This proves the remaining assertions about the final stem
and the child epochs.

To prove \eqref{eq:negative-decision}, suppose for a
contradiction that
\[
  B\in \mathcal P\cap\cyl r\cap U_{|a|}.
\]
Since the approximations to $U_{|a|}$ are increasing,
there is $s_1\geq s_0$ such that
$B\in U_{|a|,s}$ for every $s\geq s_1$. Also $B\in\mathcal P_s$
for every $s$. Thus, for every $s\geq s_1$,
\[
  B\in\mathcal P_s\cap\cyl{r_s(a,i)}\cap U_{|a|,s}.
\]
The test \eqref{eq:positive-test} is therefore positive
at every visit to $(a,i)$ after this point. The slot
has no current entry, and the fair schedule visits it
infinitely often. At such a visit the construction
must make it act, recording an event with parameters
$(a,t,i)$ and hence enumerating $i$ into $S_{a,t}$.
This contradiction proves \eqref{eq:negative-decision}.
\end{proof}

\section{Generalized lowness and specialness}\label{sec:GL}

The terms \emph{stopped} and \emph{complete} refer to the parsing in
Lemma~\ref{lem:paths}. In the decoding procedure below, $n$ counts
decoding steps, and $p_n,t_n$ abbreviate the final values
$p(a_n),\ep(a_n)$. They should not be confused with stage-$n$
approximations from Section~\ref{sec:construction}. We use the same
local notation for the coding procedure in Section~\ref{sec:coding}.

\begin{lemma}[Decoding a complete path]\label{lem:decoder}
For a complete $B\in\mathcal P$, the oracle $B\join0'$ computes its sequence
of addresses, final roots, final epochs, exit indices, and coding bits.
\end{lemma}

\begin{proof}
We describe a single oracle procedure for all complete paths,
with the computable tree $T$ and the computable construction
fixed throughout.

At step $n$, the procedure has a triple $(a_n,p_n,t_n)$ such
that $a_n$ is the address at level $n$ in the parsing of $B$,
\[
  p_n=p(a_n),\qquad t_n=\ep(a_n),\qquad p_n\prec B.
\]
Initially this holds with
\[
  (a_0,p_0,t_0)=(\emptyseq,\emptyseq,0).
\]

Suppose the current triple is known. The tail of $B$ after
$p_n$ does not belong to $\mathcal Q=[T]$, since $B$ is complete.
Using $B$ and the computable tree $T$, find the shortest
prefix of this tail which lies outside $T$. This is an exit,
hence a member of the set $E$ defined in
Section~\ref{sec:exits}. Recall that we fixed there a computable
enumeration without repetitions
\[
  E=\{\eta_i:i\in\N\}.
\]
Search this enumeration to find the unique index $i_n$.

Since the family $(S_{a,t})_{a,t}$ is uniformly c.e., the
oracle $0'$ decides whether $i_n\in S_{a_n,t_n}$, uniformly
in the finite parameters $a_n,t_n,i_n$.
For a positive answer, search the historical record for the
unique event
\[
  (a_n,t_n,i_n,\rho_n,u_n).
\]
This search terminates. For a negative answer, put
\[
  \rho_n=p_n\concat\eta_{i_n},\qquad u_n=t_n.
\]
By Lemma~\ref{lem:decision}, in either case
\[
  \rho_n=g(a_n,i_n),\qquad
  \ep(a_n\concat(i_n,b))=u_n
  \quad(b<2).
\]

By \eqref{eq:paths-exit-containment}, applied to the exit $i_n$,
$B$ extends the final stem $\rho_n$.
Read
\[
  b_n=B(|\rho_n|)
\]
and define
\[
  a_{n+1}=a_n\concat(i_n,b_n),\qquad
  p_{n+1}=\rho_n\concat b_n,\qquad
  t_{n+1}=u_n.
\]
The final-root recursion gives
$p_{n+1}=p(a_{n+1})$, and the preceding application of
Lemma~\ref{lem:decision} gives
$t_{n+1}=\ep(a_{n+1})$. The parsing rule identifies
$a_{n+1}$ as the next address of $B$, and
$p_{n+1}\prec B$. Thus the invariant is preserved.

Every search in each step terminates for a complete $B$.
Running finitely many steps computes the data at any
specified level, uniformly from $B\join0'$.
The procedure never tests whether an epoch is final, nor is it
given any stabilisation stage. Finality is preserved
inductively from the initial triple by
Lemma~\ref{lem:decision}.
\end{proof}

\begin{proposition}\label{prop:allGL}
Every member of $\mathcal P$ belongs to $\GL$.
\end{proposition}

\begin{proof}
Fix $B\in\mathcal P$. By Lemma~\ref{lem:paths}, it is either stopped
or complete.
Suppose first that $B=p(a)\concat R$ is stopped, where
$R\in \mathcal Q$. Since $p(a)$ is finite, adding or deleting this fixed
prefix is computable, so $B\eqT R$ and therefore $B'\eqT R'$.
Since $\mathcal Q\subseteq\GL$,
\[
  B'\eqT R'\eqT R\join0'\eqT B\join0'.
\]

Now suppose that $B$ is complete. Given $e$, use
Lemma~\ref{lem:decoder} to compute its address $a_{2e}$,
final epoch $t_{2e}$, and exit index $i_{2e}$ at level $2e$,
using $B\join0'$. By the parsing rule,
\[
  B\in \mathcal P\cap
  \cyl{p(a_{2e})\concat\eta_{i_{2e}}}.
\]
Therefore \eqref{eq:jumpopens} and
Lemma~\ref{lem:decision} give
\[
  e\in B'
  \quad\Longleftrightarrow\quad
  B\in U_{2e}
  \quad\Longleftrightarrow\quad
  i_{2e}\in S_{a_{2e},t_{2e}}.
\]
Since the family $(S_{a,t})_{a,t}$ is uniformly c.e.,
the last condition is decidable using $0'$ once its
finite parameters have been computed. Thus
$B'\leT B\join0'$.
Conversely, $B\leT B'$ and $0'\leT B'$ for every real $B$,
so $B\join0'\leT B'$.
Hence $B'\eqT B\join0'$ in this case as well.
\end{proof}

\begin{remark}\label{rem:pointwise}
The conclusion is asserted pointwise; no single reduction
uniform over all $B\in\mathcal P$ is claimed.
On complete paths the decoder above is uniform. If
$B=p(a)\concat R$ is stopped, the reduction may depend on the fixed
finite prefix $p(a)$ and on a reduction witnessing
$R'\leT R\join0'$ for the particular tail $R\in \mathcal Q$.
No effective procedure distinguishing stopped paths from complete
paths is required.
\end{remark}

\begin{proposition}\label{prop:special}
The class $\mathcal P$ is special.
\end{proposition}

\begin{proof}
The class $\mathcal P$ is $\Pi^0_1$ by construction and is nonempty
by Lemma~\ref{lem:paths}. Every stopped member has the form
$B=p(a)\concat R$, where $R\in \mathcal Q$. Removing this fixed
finite prefix shows that $R\leT B$, so $B$ is noncomputable
because $\mathcal Q$ is special.

Suppose for a contradiction that a complete member $B\in\mathcal P$
is computable. Choose $e$ such that $\varphi_e$ is its
total binary characteristic function. By
\eqref{eq:diagonalopens},
\[
  U_{2e+1}=\Cantor\setminus\{B\}.
\]
Let $a$ be the address of $B$ at level $2e+1$, let
$t=\ep(a)$, and let $i$ be the corresponding exit index.
Put
$r=p(a)\concat\eta_i$. By the parsing rule,
\[
  B\in \mathcal P\cap\cyl r.
\]

If $i\in S_{a,t}$, \eqref{eq:positive-decision} would give
$B\in U_{2e+1}$, which is false. Thus $i\notin S_{a,t}$.
By \eqref{eq:negative-decision} and the displayed identity
for $U_{2e+1}$,
\[
  \mathcal P\cap\cyl r\subseteq\{B\}.
\]

On the other hand, Lemma~\ref{lem:decision} gives
$g(a,i)=r$. Since $r\prec B$, the next bit of $B$ after $r$
is $B(|r|)$. Let
\[
  b=1-B(|r|).
\]
The final-root recursion gives
\[
  p(a\concat(i,b))=g(a,i)\concat b=r\concat b.
\]
By \eqref{eq:limit-copies},
\[
  r\concat b\concat \mathcal Q
  =p(a\concat(i,b))\concat \mathcal Q
  \subseteq \mathcal P\cap\cyl r.
\]
Since $\mathcal Q$ is nonempty, choose $R\in \mathcal Q$ and put
\[
  C=r\concat b\concat R.
\]
Then $C\in \mathcal P\cap\cyl r$, while
\[
  C(|r|)=b\ne B(|r|),
\]
so $C\ne B$. This contradicts
$\mathcal P\cap\cyl r\subseteq\{B\}$.
\end{proof}

\section{The Join Property and its local form}\label{sec:coding}

The goal of this section is to turn the permanent event record into the
coding required by the Join Property. The key point is that, along a
suitably chosen complete path, the coding can be decoded both from
$B\join Z$ and from $B\join0'$. We first establish an agreement lemma
used to select the exits, then carry out the coding from an arbitrary
final root. The availability of final roots inside every nonempty
relatively basic open subclass yields the local form of the Join
Property and completes the proof of the extension theorem.

For a real $Z$, write $\overline Z=\N\setminus Z$ for its complement;
as characteristic functions, $\overline Z(n)=1-Z(n)$.

\begin{lemma}[Agreement with a c.e. set]\label{lem:agreement}
Let $Z$ be noncomputable. One can choose
$\widetilde Z\in\{Z,\overline Z\}$ so that
$\overline{\widetilde Z}$ is not c.e. For every c.e. set $S$, there
is an $i$ satisfying
\begin{equation}\label{eq:agreement}
  i\in\widetilde Z\quad\Longleftrightarrow\quad i\in S.
\end{equation}
Given an index of $S$, such an $i$ can be found using
$\widetilde Z\join0'$.
\end{lemma}

\begin{proof}
Since $Z$ is noncomputable, $Z$ and $\overline Z$ cannot both be c.e.
If $\overline Z$ is not c.e., take $\widetilde Z=Z$; otherwise
$\overline Z$ is c.e., so $Z$ is not c.e., and we take
$\widetilde Z=\overline Z$. In either case
$\overline{\widetilde Z}$ is not c.e.

If \eqref{eq:agreement} failed for every $i$, then
$S=\overline{\widetilde Z}$, contradicting the fact that
$\overline{\widetilde Z}$ is not c.e.

Given an index for $S$, the oracle $0'$ decides membership in $S$
uniformly. Using also $\widetilde Z$, search for the least $i$ on
which $S$ and $\widetilde Z$ agree. The preceding argument guarantees
that this search terminates.
\end{proof}

The choice of $\widetilde Z$ is made once and for all and need not
be uniform in $Z$. Once this choice is fixed, the search for an
agreement is uniform in an index of $S$.

\begin{proposition}[Coding from a final root]\label{prop:coding}
Fix an address $a_*$ and reals $A,Z$ with $A\geT Z\join0'$ and
$Z\gtT0$. There is a complete path
$B\in \mathcal P\cap\cyl{p(a_*)}$ satisfying \eqref{eq:JP}.
\end{proposition}

\begin{proof}
Fix $\widetilde Z$ as in Lemma~\ref{lem:agreement}.
This orientation need not be chosen uniformly from $Z$;
it is one fixed binary choice, and $\widetilde Z\eqT Z$.
Fix also the finite parameters
\[
  a_*,\qquad p(a_*),\qquad \ep(a_*),
\]
together with an index for a Turing reduction witnessing
$Z\join0'\leT A$.
These finite parameters and this index may be built into the
relevant oracle programs. The initial root and its epoch are treated here as fixed
parameters; their $0'$-uniform recovery is discussed in
Section~\ref{sec:conclusion}.

\smallskip
\noindent\emph{Construction of the path.}
Using $A$, we construct triples $(a_n,p_n,t_n)$ satisfying
\[
  a_*\preceq a_n,\qquad |a_n|=|a_*|+n,\qquad
  p_n=p(a_n),\qquad t_n=\ep(a_n).
\]
Start with
\[
  (a_0,p_0,t_0)=(a_*,p(a_*),\ep(a_*)).
\]
Given the current triple, choose the least $i_n$ such that
\begin{equation}\label{eq:coding-agreement}
  i_n\in\widetilde Z
  \quad\Longleftrightarrow\quad
  i_n\in S_{a_n,t_n}.
\end{equation}

Lemma~\ref{lem:agreement} guarantees existence. Since
$\widetilde Z\eqT Z$, we have
$\widetilde Z\join0'\leT A$. As the family
$(S_{a,t})_{a,t}$ is uniformly c.e., the least such $i_n$
can therefore be found using $A$.

If $i_n\in S_{a_n,t_n}$, search the historical record for
the unique event $(a_n,t_n,i_n,\rho_n,u_n)$.
Otherwise put
\[
  \rho_n=p_n\concat\eta_{i_n},\qquad u_n=t_n.
\]
By Lemma~\ref{lem:decision}, in either case
\[
  \rho_n=g(a_n,i_n),\qquad
  \ep(a_n\concat(i_n,b))=u_n
  \quad(b<2).
\]
Set
\begin{equation}\label{eq:coding-next}
  a_{n+1}=a_n\concat(i_n,A(n)),\qquad
  p_{n+1}=\rho_n\concat A(n),\qquad
  t_{n+1}=u_n.
\end{equation}
The final-root recursion gives
$p_{n+1}=p(a_{n+1})$, and the preceding application of
Lemma~\ref{lem:decision} gives
$t_{n+1}=\ep(a_{n+1})$. Thus the induction invariant
is preserved.

Moreover,
\[
  p_n\concat\eta_{i_n}
  \preceq\rho_n\prec p_{n+1}.
\]
Since every exit is nonempty, the roots are nested and
their lengths increase by at least two at each step.
Let
\[
  B=\bigcup_n p_n.
\]
Adjoining the proper prefixes of $a_*$, in increasing order of
length, to the sequence $(a_n)_{n\in\N}$ gives an infinite
address chain starting at $\emptyseq$.

The final roots of these additional
addresses are prefixes of $p(a_*)=p_0$, so their inclusion
does not change the union. Lemma~\ref{lem:paths} therefore
shows that $B$ is complete and
\[
  B\in \mathcal P\cap\cyl{p(a_*)}.
\]

\smallskip
\noindent\emph{The reduction $B\leT A$.}
With the fixed auxiliary data specified above, the sequence
$(p_n)_{n\in\N}$ is computable from $A$.

To compute $B(k)$,
construct roots until $|p_n|>k$ and return $p_n(k)$.
This terminates because the lengths are unbounded. Hence
\begin{equation}\label{eq:BbelowA}
  B\leT A.
\end{equation}
By construction, for every $n$,
\[
  B(|\rho_n|)=A(n).
\]

\smallskip
\noindent\emph{Decoding from $B\join Z$.}
Use the fixed orientation to answer questions about
$\widetilde Z$, and start with
$(a_*,p(a_*),\ep(a_*))$.
Inductively suppose the triple $(a_n,p_n,t_n)$ used in
the coding has been recovered.
Using $B$ and the computable tree $T$, find the shortest prefix
of the tail after $p_n$ which lies outside $T$. This is the unique
exit $\eta_{i_n}$, and its index $i_n$ is recovered from the fixed
computable enumeration of $E$.
Thus the decoder recovers from $B$ the exit selected during the
coding; it does not repeat the least-agreement search.
If $i_n\in\widetilde Z$, search the historical record for
the event with parameters $(a_n,t_n,i_n)$ and retrieve
$(\rho_n,u_n)$.
This search terminates because \eqref{eq:coding-agreement} gives
$i_n\in S_{a_n,t_n}$, so the corresponding event exists.

If $i_n\notin\widetilde Z$,
the same equivalence guarantees that no such event exists,
so use
\[
  \rho_n=p_n\concat\eta_{i_n},\qquad u_n=t_n.
\]
In either case these are exactly the data used in the
coding. Read $b_n=B(|\rho_n|)=A(n)$ and update the triple to
\[
  (a_n\concat(i_n,b_n),\rho_n\concat b_n,u_n).
\]
By \eqref{eq:coding-next}, this is the next coding triple.
Thus the induction continues and
$A\leT B\join Z$.

\smallskip
\noindent\emph{Decoding from $B\join0'$.}
Start with the same fixed initial triple and recover the exit
$i_n$ from $B$ exactly as above.

Instead of querying
$\widetilde Z$, use $0'$ to decide
$i_n\in S_{a_n,t_n}$.
The positive or negative answer recovers the same
$\rho_n,u_n$ as above, after which the coding bit and
the next triple are recovered in the same way.
This decoder does not use $Z$.
Consequently $A\leT B\join0'$.
Together with $B\leT A$ and $Z,0'\leT A$, the two decoding
reductions give
\[
  A\eqT B\join Z\eqT B\join0'.
\]
By Proposition~\ref{prop:allGL},
\[
  B'\eqT B\join0',
\]
and hence all the equivalences in \eqref{eq:JP} hold.
The agreement \eqref{eq:coding-agreement} is required
only along the path selected for the given $A,Z$,
not on every member of $\mathcal P$.
\end{proof}

\begin{lemma}[Final roots are locally available]\label{lem:localroot}
If $\mathcal P\cap\cyl{\sigma}\ne\varnothing$, some final root $p(a)$ extends
$\sigma$.
\end{lemma}

\begin{proof}
Choose $B\in \mathcal P\cap\cyl{\sigma}$.

If $B$ is complete, let $(a_n)_{n\in\N}$ be its address
chain. The roots $p(a_n)$ are prefixes of $B$ and have
unbounded lengths. Choose $n$ with
$|p(a_n)|\geq|\sigma|$. Since $\sigma$ and $p(a_n)$ are
both prefixes of $B$, we have $\sigma\preceq p(a_n)$.

Suppose instead that
\[
  B=p(a_0)\concat R,\qquad R\in \mathcal Q.
\]
If $\sigma\preceq p(a_0)$, take $a=a_0$.
Otherwise, since $\sigma$ and $p(a_0)$ are both prefixes of $B$
and $\sigma\npreceq p(a_0)$, we have $p(a_0)\prec\sigma$.
Hence there is a nonempty finite string $\delta$ such that
\[
  \sigma=p(a_0)\concat\delta,\qquad \delta\prec R.
\]
In particular, $R\in\mathcal Q\cap\cyl{\delta}$, so
Lemma~\ref{lem:exits} gives an exit $\eta_i$ extending
$\delta$. By the final stem and root identities,
\[
  \sigma
  =p(a_0)\concat\delta
  \preceq p(a_0)\concat\eta_i
  \preceq g(a_0,i)
  \prec p(a_0\concat(i,0)).
\]
Thus $a=a_0\concat(i,0)$ is a suitable address.
\end{proof}

\begin{proof}[Completion of the proof of
Theorem~\ref{thm:extension}]
The class $\mathcal P$ is $\Pi^0_1$ by construction, and
$\mathcal Q\subseteq\mathcal P$ by Lemma~\ref{lem:paths}.
Propositions~\ref{prop:allGL} and \ref{prop:special}
show that $\mathcal P\subseteq\GL$ and that $\mathcal P$ is special.

Fix a finite binary string $\sigma$ such that
$\mathcal P\cap\cyl{\sigma}\ne\varnothing$.
By Lemma~\ref{lem:localroot}, choose an address $a$ with
$\sigma\preceq p(a)$. Hence
$\cyl{p(a)}\subseteq\cyl{\sigma}$.

Let $A,Z$ be arbitrary reals satisfying
$A\geT Z\join0'$ and $Z\gtT0$.
Apply Proposition~\ref{prop:coding} with initial address
$a_*=a$. It provides a complete path
\[
  B\in \mathcal P\cap\cyl{p(a)}
       \subseteq \mathcal P\cap\cyl{\sigma}
\]
such that
\[
  A\eqT B'\eqT B\join0'\eqT B\join Z.
\]
Thus $\mathcal P\cap\cyl{\sigma}$ has the Join Property.
Since $\sigma$ was arbitrary, this proves the local
conclusion and completes the theorem.
\end{proof}

\section{Separation and a nontrivial pseudojump}\label{sec:consequences}

We now derive the separation between the Join Property and pseudojump
inversion. The decisive observation is that no class contained in
$\mathrm{GL}_1$ can have the Pseudojump Inversion Property. This
completes the proof of the main theorem. We then record a stronger
form of the obstruction, showing that the failure already occurs for
a fixed nontrivial pseudojump operator which strictly raises every
input degree.

\begin{proposition}\label{prop:GLobstruction}
No class $\mathcal K\subseteq\GL$ has the Pseudojump Inversion Property.
\end{proposition}

\begin{proof}
For every $B\in\GL$ we have $0'\nleT B$. Otherwise,
\[
  B'\eqT B\join0'\eqT B,
\]
contradicting strictness of the jump.

Let $\mathcal K\subseteq\GL$, and suppose that $\mathcal K$ has the
Pseudojump Inversion Property. Fix an index
$e_{\varnothing}$ such that
\[
  W_{e_{\varnothing}}^X=\varnothing
  \qquad\text{for every }X.
\]
Then $J_{e_{\varnothing}}(B)\eqT B$ for every $B$.
Applying \eqref{eq:PJI} with $A=0'$ and
$e=e_{\varnothing}$ gives a member $B\in \mathcal K$ satisfying
\[
  0'\eqT J_{e_{\varnothing}}(B)\eqT B.
\]
This contradicts $0'\nleT B$, since $B\in\GL$.
\end{proof}

\begin{proof}[Proof of Theorem~\ref{thm:main}]
Take a special $\Pi^0_1$ class $\mathcal Q\subseteq\GL$, as provided by
\cite[Theorem 6.2, parts 2 and 4]{JS}, and apply
Theorem~\ref{thm:extension}.
The resulting $\mathcal P$ has the Join Property, including its asserted
local form, but does not have the Pseudojump Inversion Property by
Proposition~\ref{prop:GLobstruction}.
\end{proof}

The use of the pseudojump operator $J_{e_{\varnothing}}$ above is
sufficient, but it is not essential. The following observation isolates the generalized-lowness
argument used in \cite[proof of Theorem 6.4]{JS}.

\begin{proposition}\label{prop:properpseudo}
Let $J$ be a pseudojump operator such that
\[
  X\ltT J(X)\quad\text{and}\quad J(X)'\eqT X'
  \qquad\text{for every }X.
\]
Then $0'\nleT J(B)$ for every $B\in\GL$.
\end{proposition}
\begin{proof}
If $0'\leT J(B)$, the definition of a pseudojump also gives
$B\leT J(B)$. Hence
\[
  J(B)'\eqT B'\eqT B\join0'\leT J(B),
\]
contradicting strictness of the jump.
\end{proof}

Such an operator is used in \cite[proof of Theorem 6.4]{JS}.
Consequently, for the constructed class $\mathcal P$, pseudojump inversion
already fails for this single fixed operator: taking $A=0'$, there is
no $B\in\mathcal P$ for which
\[
  0'\eqT J(B)\eqT B\join0'.
\]
Indeed, Proposition~\ref{prop:properpseudo} rules out even the first
equivalence. Thus the obstruction persists for a pseudojump operator
which strictly raises every input degree. This observation is not
needed for the main separation.

\section{Concluding remarks and further questions}\label{sec:conclusion}

We have shown that the Join Property and the Pseudojump Inversion
Property can be separated on a special $\Pi^0_1$ class. In fact, the
separation can be obtained under the additional restriction that every
member of the class is generalized low, while the Join Property holds
locally on every nonempty relatively basic open subclass. The
extension theorem shows that this is not an isolated construction:
every special $\Pi^0_1$ class contained in $\mathrm{GL}_1$ can be
enlarged, without leaving $\mathrm{GL}_1$, to a class with the local
Join Property.

The role of generalized lowness suggests a first direction for further
study. In the present argument, $\mathrm{GL}_1$ serves both as a class
that can be preserved by the construction and as an obstruction to
pseudojump inversion. It would be natural to ask for which other
degree-theoretic classes $\mathcal C$ an analogous extension theorem is
possible: given a special $\Pi^0_1$ class
$\mathcal Q\subseteq\mathcal C$, when can one find a special
$\Pi^0_1$ extension $\mathcal P$ still contained in $\mathcal C$ and
having the local Join Property?  One may also ask how far the
generalized-lowness hypothesis can be weakened while still obtaining a
separation between the two properties.

A second issue concerns uniformity. For the fixed computable
construction, the functions $a\mapsto p(a)$ and
$a\mapsto\ep(a)$ are uniformly computable in $0'$, since
their stage approximations are computable and eventually
constant at each input. Hence, given $\sigma$ with
$\mathcal P\cap\cyl{\sigma}\ne\varnothing$, one can use $0'$
to search for an address $a$ such that $\sigma\preceq p(a)$;
Lemma~\ref{lem:localroot} guarantees that this search terminates.
The coding proof nevertheless fixes auxiliary data: an
orientation $\widetilde Z\in\{Z,\overline Z\}$ with
$\overline{\widetilde Z}$ not c.e., an index witnessing
$Z\join0'\leT A$, and the finite initial data used by the
decoders. In particular, the decoder from $B\join Z$ is
not given $0'$ as an additional oracle. It is therefore
natural to investigate stronger uniform local formulations
in which the permitted oracle inputs and reduction indices
are specified explicitly, and to determine which auxiliary
choices can be eliminated.

Finally, Appendix~\ref{app:fibres} shows that every jump-inversion
fibre above $0'$ of a class with the Join Property is countably infinite.
It is natural to investigate the order-theoretic structure of
these fibres and to determine which further constraints follow
from additional hypotheses on the class.

More generally, since pseudojump operators admit finite and transfinite
iterations, it would be interesting to investigate analogous
class-restricted inversion and join phenomena for iterated pseudojumps
and for higher levels of the corresponding hierarchies.

\appendix
\section{Infinitely many degrees in each Join Property fibre}\label{app:fibres}

This appendix is independent of the construction. It records the
additional fibre result without assuming that the Turing jump preserves
finite joins.

Write $\degT(B)$ for the Turing degree of $B$, that is, its
equivalence class under $\eqT$.

For an arbitrary class
$\mathcal R\subseteq\Cantor$ and $A\geT0'$, put
\[
 \mathcal F_{\mathcal R}(A)=
 \{\degT(B):B\in\mathcal R\ \&\
   B'\eqT B\join0'\eqT A\}.
\]

Thus $\mathcal F_{\mathcal R}(A)$ is the generalized-low
part of the ordinary jump-inversion fibre.
Recall that a real $G$ is \emph{$1$-generic relative to $H$} if, for
every $H$-c.e. set of strings $W\subseteq\Strings$, some prefix
$\sigma\prec G$ satisfies
\[
  \sigma\in W
  \quad\text{or}\quad
  \{\tau\in W:\sigma\preceq\tau\}=\varnothing.
\]
These alternatives say that $G$ \emph{meets} or \emph{avoids} $W$,
respectively. A set of strings is \emph{dense} if every finite binary
string has an extension in it. Finally, a $\Sigma^0_1(H)$ question
is an existential question over an $H$-computable predicate, with
finite parameters allowed; such questions are decidable using $H'$.

\begin{lemma}\label{lem:finitegeneric}
Suppose $m\ge1$ and $B_i'\leT A$ for each $i<m$. There is a
noncomputable $G\leT A$ which is $1$-generic relative to each $B_i$.
Moreover,
\[
  (B_i\join G)'\eqT B_i'\join G\qquad(i<m).
\]
\end{lemma}

\begin{proof}
\noindent\emph{Construction of $G$.}
For each $i<m$, fix an index witnessing $B_i'\leT A$.
Since the family is finite, these indices can be built into
a single oracle program. Also $B_i\leT B_i'\leT A$ for
every $i<m$.

View $W_e^{B_i}$ as a set of finite binary strings under a
fixed computable coding. Fix a computable enumeration without repetitions of all
pairs $(i,e)$ with $i<m$ and $e\in\N$.
Using $A$, construct a nested sequence of strings
$(\sigma_s)_{s\in\N}$, starting with the empty string.

At the step for $(i,e)$, use $B_i'$ to decide whether
\[
  \exists\tau\succeq\sigma_s\;
  \bigl(\tau\in W_e^{B_i}\bigr).
\]
This is a $\Sigma^0_1(B_i)$ question, uniformly in the
finite parameters. If the answer is positive, enumerate
$W_e^{B_i}$ until such a $\tau$ appears. If the answer is
negative, put $\tau=\sigma_s$. In either case set
\[
  \sigma_{s+1}=\tau\concat0.
\]
All these operations are computable in $A$.

Let $G=\bigcup_s\sigma_s$. The lengths increase at every
step, so $G\leT A$. If a requirement receives a positive
answer, the chosen $\tau\in W_e^{B_i}$ remains a prefix of
$G$. If it receives a negative answer, $\sigma_s$ remains
a prefix of $G$ with no extension in $W_e^{B_i}$.
Thus $G$ is $1$-generic relative to each $B_i$.

\smallskip
\noindent\emph{Noncomputability.}
For every $i<m$, this implies $G\nleT B_i$. Indeed, if
$G\leT B_i$, then
\[
  D_G=\{\sigma\in2^{<\N}:
          \exists k<|\sigma|\;
          [\sigma(k)\ne G(k)]\}
\]
would be a dense $B_i$-computable set of strings.
No prefix of $G$ belongs to $D_G$, and density prevents
$G$ from avoiding it, contrary to relative genericity.
Since $m\geq1$, fixing any $i<m$ gives
$G\nleT B_i$, and hence $G$ is noncomputable.

\smallskip
\noindent\emph{The jump of $B_i\join G$.}
Finally, fix $i<m$. For each $e$, let
\[
  C_{i,e}=
  \{\tau\in2^{<\N}:
       \Phi_e^{B_i\join\tau}(e)\downarrow\}.
\]
Here $\Phi_e^{B_i\join\tau}(e)\downarrow$ means that the computation
converges using only bits of the finite second oracle below
$|\tau|$.
Consequently, this finite computation persists, with the same
outcome, for every infinite extension of $\tau$ used as the
second oracle.
The set $C_{i,e}$ is upward closed and uniformly c.e. in $B_i$.

Since $G$ is $1$-generic relative to $B_i$, it either meets
$C_{i,e}$ or has a prefix with no extension in $C_{i,e}$.
Using $B_i'\join G$, examine successive prefixes $\sigma$
of $G$. The oracle $B_i'$ decides both whether
$\sigma\in C_{i,e}$ and whether some extension of $\sigma$
belongs to $C_{i,e}$. If $\sigma\in C_{i,e}$, answer that
$e\in(B_i\join G)'$. If no extension of $\sigma$ belongs
to $C_{i,e}$, answer that $e\notin(B_i\join G)'$.
By the preceding genericity alternative, one of these cases
is eventually found.

The positive answer is correct because the finite computation
witnessed by the prefix $\sigma$ persists along $G$. The negative answer is
correct because any convergence with oracle $B_i\join G$
would be witnessed by a sufficiently long prefix of $G$
extending $\sigma$. Hence
\[
  (B_i\join G)'\leT B_i'\join G.
\]
Conversely, monotonicity of the jump gives
$B_i'\leT(B_i\join G)'$, and
$G\leT B_i\join G\leT(B_i\join G)'$.
Therefore
\[
  (B_i\join G)'\eqT B_i'\join G.
\]
\end{proof}

\begin{proposition}\label{prop:infinitefibres}
If $\mathcal R$ has the Join Property, then
$\mathcal F_{\mathcal R}(A)$ is countably infinite for every
$A\geT0'$.
\end{proposition}

\begin{proof}
Fix $A\geT0'$.
Since $0'\gtT0$ and $0'\join0'\eqT0'\leT A$, the Join Property
applied with $Z=0'$ gives a real $C\in \mathcal R$ such that
\[
  C'\eqT C\join0'\eqT A.
\]
Thus $\mathcal F_{\mathcal R}(A)$ is nonempty.

If $B$ represents a degree in this fibre, then
$B\leT B'\eqT A$, and hence $B\leT A$.
Since only countably many reals are computable in a fixed oracle,
$\mathcal F_{\mathcal R}(A)$ is at most countable.

Suppose for a contradiction that the fibre is finite.
Choose representatives $B_0,\ldots,B_{m-1}\in \mathcal R$, with
$m\geq1$, for all its degrees, so that
\[
  B_i'\eqT B_i\join0'\eqT A
  \qquad(i<m).
\]
Apply Lemma~\ref{lem:finitegeneric} to this finite family.
It gives a noncomputable $G\leT A$ such that
\[
  (B_i\join G)'
  \eqT B_i'\join G
  \eqT A
  \qquad(i<m),
\]
where the last equivalence uses $B_i'\eqT A$ and
$G\leT A$.

For each $i<m$, we have $A\nleT B_i\join G$.
Indeed, otherwise
\[
  (B_i\join G)'\leT A\leT B_i\join G,
\]
contradicting strictness of the jump.

Since $G$ is noncomputable and $G\join0'\leT A$,
the Join Property, now applied with $Z=G$, supplies
a real $B\in \mathcal R$ satisfying
\[
  A\eqT B'\eqT B\join0'\eqT B\join G.
\]
Thus $\degT(B)\in\mathcal F_{\mathcal R}(A)$, so $B\eqT B_j$
for some $j<m$. Consequently,
\[
  A\eqT B\join G\eqT B_j\join G,
\]
contrary to $A\nleT B_j\join G$.

Hence the fibre is infinite, and therefore countably infinite.
\end{proof}

Consequently, if $\mathcal R$ has the Join Property and $A\geT0'$,
the ordinary jump-inversion fibre
$\{\degT(B):B\in\mathcal R\ \&\ B'\eqT A\}$
is also countably infinite: it contains $\mathcal F_{\mathcal R}(A)$,
and every one of its representatives is computable in $A$.

\begin{remark}
The genericity construction uses the reductions $B_i'\leT A$
separately. In particular, it does not assume that the jump of
$B_0\join\cdots\join B_{m-1}$ is computable in $A$.
\end{remark}

\section*{Acknowledgements}
The main results were discovered by ChatGPT 6 Pro (OpenAI). The author 
takes full responsibility for the mathematical content.


\begin{thebibliography}{99}


\bibitem{CDJL}
R.~Coles, R.~Downey, C.~G. Jockusch, Jr., and G.~LaForte,
\emph{Completing pseudojump operators},
Ann. Pure Appl. Logic \textbf{136} (2005), no.~3, 297--333.


\bibitem{DG}
R.~Downey and N.~Greenberg,
\emph{Pseudo-jump inversion, upper cone avoidance, and
strong jump-traceability},
Adv. Math. \textbf{237} (2013), 252--285.

\bibitem{JS}
H.~R. Jananthan and S.~G. Simpson,
\emph{Pseudojump inversion in special r.\ b.\ $\Pi^0_1$ classes},
arXiv:2102.06135v1, 2021.
\url{https://arxiv.org/abs/2102.06135v1}.

\bibitem{JSh}
C.~G. Jockusch, Jr., and R.~A. Shore,
\emph{Pseudojump operators. I: The r.e. case},
Trans. Amer. Math. Soc. \textbf{275} (1983), 599--609.

\bibitem{JShII} C.~G. Jockusch, Jr., and R.~A. Shore,
\emph{Pseudo-jump operators. II: Transfinite iterations,
hierarchies and minimal covers},
J. Symbolic Logic \textbf{49} (1984), no.~4, 1205--1236.



\bibitem{Odifreddi}
P. Odifreddi,
\emph{Classical Recursion Theory. Vol.~II},
Studies in Logic and the Foundations of Mathematics, Vol. 143,
North-Holland, Amsterdam, 1999.

\bibitem{PR}
D.~B. Posner and R.~W. Robinson,
\emph{Degrees joining to $0'$},
J. Symbolic Logic \textbf{46} (1981), 714--722.

\bibitem{Soare}
R.~I. Soare,
\emph{Recursively Enumerable Sets and Degrees:
A Study of Computable Functions and Computably Generated Sets},
Perspectives in Mathematical Logic,
Springer-Verlag, Berlin, 1987.

\end{thebibliography}
\end{document}